%% file: src.tex
\documentclass[11pt]{article}

\usepackage{amsmath,amssymb}
\usepackage{amsfonts,mathrsfs}
\usepackage{amsthm,ascmac}
\usepackage{amscd,mathtools}

\mathtoolsset{showonlyrefs=true}

\usepackage[nameinlink]{cleveref}
\usepackage{autonum}
\usepackage{comment}
\usepackage{xspace}

\theoremstyle{definition}
\usepackage{aliascnt}

\newcommand*{\cnewtheorem}[2]{%
    \newaliascnt{#1}{cnt}%
    \newtheorem{#1}[#1]{#2}%
    \crefname{#1}{#2}{}%
    \aliascntresetthe{#1}%
    \expandafter\newcommand\csname #1autorefname\endcsname{#2}%
    \newtheorem*{#1*}{#2}%
}

\cnewtheorem{theorem}{Theorem}
\cnewtheorem{lemma}{Lemma}
\cnewtheorem{proposition}{Proposition}
\cnewtheorem{corollary}{Corollary}
\cnewtheorem{definition}{Definition}
\cnewtheorem{remark}{Remark}
\cnewtheorem{example}{Example}
\cnewtheorem{conjecture}{Conjecture}
\cnewtheorem{question}{Question}
\cnewtheorem{problem}{Problem}
\cnewtheorem{notation}{Notation}

\cnewtheorem{claim}{Claim}
\numberwithin{equation}{cnt}

\renewcommand{\phi}{\varphi}
\newcommand{\ep}{\varepsilon}

\newcommand{\bQ}{\mathbb{Q}}

\newcommand{\cL}{\mathcal{L}}

\newcommand{\cP}{\mathcal{P}}

\newcommand{\cU}{\mathcal{U}}
\newcommand{\cV}{\mathcal{V}}

\newcommand{\rT}{\mathrm{T}}

\newcommand{\cstar}{{\mathrm{C}^{\ast}}}
\newcommand{\conti}{{\mathfrak c}}

\DeclareMathOperator{\cf}{cf}

\DeclareMathOperator{\Th}{Th}

\DeclareMathOperator{\clop}{clop}

\newcommand{\set}[3][\relmiddle|]{\left\{#2 #1 #3\right\}}
\newcommand{\norm}[1]{\left\lVert {#1} \right\rVert}
\newcommand{\abs}[1]{\left\lvert {#1} \right\rvert}

\newcommand{\deq}{\mathrel{:=}}

\let\LL\L
\newcommand{\Los}{\LL o\'{s}\xspace}
\newcommand{\LST}{Löwenheim--Skolem--Tarski\xspace}
\newcommand{\KS}{Keisler--Shelah\xspace}
\newcommand{\CH}{continuum hypothesis\xspace}

\newcommand{\trel}{\vartriangleleft}
\newcommand{\treleq}{\trianglelefteq}

\makeatletter
\newcommand{\dotminus}{\mathbin{\text{\@dotminus}}}

\title{On the isomorphism problem for ultraproducts of $\mathrm{C}^*$-algebras in continuous model theory}
\author{Akihiko Arai\thanks{Graduate School of Science and Engineering Chiba University}}
\date{\today}

\begin{document}

\maketitle

\input{files/abstruct}

%\tableofcontents

\input{files/refs}

\input{files/introduction}

\input{files/preliminaries}

\input{files/main_result}

\input{files/cstar}

\input{files/topology}

\input{files/conclusion}

\bibliographystyle{plain}
\bibliography{references}

\end{document}

%% file: files/abstruct.tex
\begin{abstract}
    In classical model theory, the Keisler--Shelah theorem establishes a fundamental connection between the elementary equivalence of structures and the isomorphism of their ultrapowers.
    Motivated by this, one may ask whether an analogous relationship holds in the framework of continuous model theory,
    which naturally encompasses metric structures such as $\mathrm{C}^\ast$-algebras.
    In this paper, we investigate the isomorphism problem for ultraproducts of operator algebras from a model-theoretic perspective.
    We prove that, assuming the negation of the continuum hypothesis, there exist two elementarily equivalent infinite-dimensional unital $\mathrm{C}^\ast$-algebras $A$ and $B$,
    whose density characters are at most $\mathfrak c$,
    such that for all non-principal ultrafilters $\mathcal U, \mathcal V$ on $\omega$,
    the ultrapowers $A^{\mathcal U}$ and $B^{\mathcal V}$ are not isomorphic.
    This result provides a continuous analogue of certain classical theorems concerning ultraproducts and demonstrates that the model-theoretic behavior of $\mathrm{C}^\ast$-algebras is closely related to set-theoretic principles such as the continuum hypothesis.
\end{abstract}

%% file: files/refs.tex
\nocite{Farah:2019}
\nocite{Bradd:2023}

\nocite{MTFMS}
\nocite{MToCA}

\nocite{FHS1:2013}
\nocite{FHS2:2014}
\nocite{FHS3:2014}

\nocite{EAGLE:2015}
\nocite{FM:2021}
\nocite{GK:2022}

\nocite{Wright:1954}
\nocite{Los:1955}
\nocite{Keisler:1961}
\nocite{Robinson:1966}
\nocite{Sakai:1962}
\nocite{McDuff:1970}
\nocite{Shelah:1971}
\nocite{Shelah:1992}
\nocite{Connes:1976}
\nocite{CK:1990}
%\nocite{Kirchberg:1993}
\nocite{Kirchberg:1995}
\nocite{GH:2001}
\nocite{GS1:2023}
\nocite{GS2:2023}

\nocite{Tsuboi:2022}
\nocite{Goto:2024}

\nocite{CK:1966}
\nocite{Blackadar:2006}
\nocite{JS:1999}

%% file: files/introduction.tex
\section{Introduction}

The ultraproduct construction is a significant object of study in both model theory and operator algebras.
The model-theoretic notion of ultraproducts was introduced by {\Los} in \cite{Los:1955} in the 1950's,
and Robinson made use of it to develop nonstandard analysis (cf. \cite{Robinson:1966}).
Around the same time, an ultraproduct-like construction appeared in the context of operator algebras
in the work of Kaplansky and Wright \cite{Wright:1954}.
Subsequently, Sakai introduced the ultraproduct of $\mathrm{II}_1$ factors in \cite{Sakai:1962},
and McDuff \cite{McDuff:1970} discovered its significance.
Later, Connes \cite{Connes:1976} and Kirchberg \cite{Kirchberg:1995} applied this construction to the classification theory of operator algebras.
However, no essential connection between these two perspectives seems to have been recognized at the time.

In model theory, the isomorphism of ultraproducts of two structures is closely related to the notion of \emph{elementary equivalence}.
One of the earliest results in this direction is due to Keisler \cite{Keisler:1961}, who showed that,
under the assumption of the \CH, for any $\cL$-theory $\rT$ in a countable language $\cL$ and any two models $M$ and $N$ of $\rT$ of cardinality at most $\conti$,
$M$ and $N$ are elementarily equivalent (i.e., $M \equiv N$) if and only if they have isomorphic ultrapowers with respect to an ultrafilter on $\omega$.
Furthermore, Shelah \cite{Shelah:1971} (see also \cite{Shelah:1992}) generalized this result by weakening the assumption on the ultrafilter from $\omega$ to sufficiently large cardinals,
thereby showing that the theorem holds in ZFC.
This is known as the \emph{Keisler--Shelah theorem}.
Thus, the following question naturally arises.

\begin{question} \label{ques:intro}
    Assume the \CH fails.
    For which $\cL$-theories $\rT$ in a countable language $\cL$, do there exist models $M$ and $N$ of $\rT$ of size at most $\conti$ (or $\aleph_0$) such that
    $M$ and $N$ are elementarily equivalent,
    but $M^\cU$ and $N^\cV$ are not isomorphic for all non-principal ultrafilters $\cU, \cV$ on $\omega$?
\end{question}

Several concrete results, as well as sufficient conditions for \cref{ques:intro}, have been obtained in \cite{GS1:2023}, \cite{GS2:2023}, \cite{Tsuboi:2022}, and \cite{Goto:2024}.

In recent years, \emph{continuous model theory}, a generalization of classical model theory to the setting of metric structures,
has gained increasing attention as a powerful tool for analyzing ultraproducts of operator algebras.
For example, it was shown in this framework that all countable ultrapowers of a separable $\cstar$-algebra $A$ are isomorphic if and only if the \CH holds (cf. \cite[Theorem 5.1]{FHS1:2013}).

Accordingly, we are interested in the following problem,
which is an extension of \cref{ques:intro} to the framework of continuous model theory,
particularly in the case of the theory of operator algebras.
(It is mentioned in \cite[\S 4]{FHS3:2014} in the case of $\mathrm{II}_1$ factors.)
\begin{question} \label{ques:conti}
    Assume the \CH fails.
    For which $\cL$-theories $\rT$ with $\chi(\rT, \cL) \le \aleph_0$
    do there exist models $M$ and $N$ of $\rT$ of density character at most $\conti$ (or $\aleph_0$) such that $M \equiv N$,
    but $M^\cU \not\cong N^\cV$ for all non-principal ultrafilters $\cU, \cV$ on $\omega$?
\end{question}

In this paper, we study this problem in the setting of $\cstar$-algebras.
For \cref{ques:conti}, we obtain the following result:
\begin{theorem} \label{thm:main-result-intro}
    Assume the \CH fails.
    Then there exist $\cstar$-algebras $A$ and $B$
    of density character at most $\conti$ such that $A \equiv B$,
    but $A^\cU$ and $B^\cV$ are not isomorphic for any non-principal ultrafilters $\cU, \cV$ on $\omega$.
\end{theorem}

We describe the overall outline of this paper.
In Section 2, we present preliminaries, including several basic results in continuous model theory.

In Section 3, we provide a sufficient condition for \cref{ques:intro} via \cref{thm:main-result-general}.
This result applies to general continuous languages,
and in the classical setting it yields several known results; see \cref{rmk:classical-results}.

In Section 4, we prove \cref{thm:main-result-intro} based on the results of Section 3. In particular, \cref{lem:example-of-main-result} provides a concrete example of a theory satisfying \cref{thm:main-result-intro}.

In Section 5, we discuss a topological application of the results obtained in Section 4.
Finally, in Section 6, we propose several related open problems as directions for future research.

\subsection*{Acknowledgement}

I am grateful to the anonymous referee whose numerous comments and suggestions helped to improve the exposition and sharpen the results of the paper substantially.
I would like to express my deepest gratitude to my supervisor, Hiroki Matui, for his continuous guidance and encouragement.
I am also grateful to Kota Takeuchi and all members of his laboratory for their helpful comments and discussions through seminars.
Finally, I would like to thank Tatsuya Goto, who introduced this problem at the Model Theory Summer School and Workshop 2024.

%% file: files/preliminaries.tex
\section{Preliminaries}

In this section, we present several basic results in continuous model theory, without proofs.
For detailed proofs and further discussion, we refer the reader to \cite{Bradd:2023}, \cite{FHS2:2014}, \cite{MToCA}, and \cite{MTFMS}.

\begin{notation}
    \begin{itemize}
        \item We denote ordinals by $\alpha, \beta, \gamma, \dots$ and cardinals by $\kappa, \lambda, \dots$.
        \item For a cardinal $\kappa$, we denote by $\kappa^+$ the successor cardinal of $\kappa$, that is, the least cardinal greater than $\kappa$.
        \item We denote by $\omega_\alpha$ the initial ordinal of cardinality $\aleph_\alpha$.
              In particular, $\omega_0$, which is the set of all natural numbers, is simply denoted by $\omega$.
        \item We denote the cardinality of the continuum $2^{\aleph_0}$ by $\conti$.
              The \emph{continuum hypothesis} (CH) is the statement $\conti = \aleph_1$.
              Thus, $\lnot$CH is equivalent to $\aleph_1 < \conti$, and hence to $\aleph_2 \le \conti$.
        \item For a set $X$ equipped with a binary relation $R$,
              the \emph{cofinality} of $(X, R)$,
              denoted by $\cf(X, R)$,
              is the least cardinality of a cofinal subset of $(X, R)$.
        \item We denote by $\beta I \setminus I$ the set of all non-principal ultrafilters on a set $I$.
    \end{itemize}
\end{notation}

We use the notation of continuous model theory introduced in \cite{Bradd:2023}.
From now on, unless otherwise specified,
we assume that $\cL$ is a language
(in the sense of continuous model theory),
and that $\rT$ is a complete theory with $\chi(\rT, \cL) \le \aleph_0$.

Many fundamental results in classical model theory admit natural analogues in continuous model theory.
For the reader's convenience, we list below several relevant theorems together with references:
\begin{itemize}
    \item \Los' theorem (see \cite[Theorem 5.4]{MTFMS} for a complete proof),
    \item The downward \LST theorem (see \cite[Theorem 5.6]{Bradd:2023} and \cite[Theorem 4.6]{FHS2:2014}),
    \item The Tarski's elementary chain theorem (cf. \cite[Proposition 5.10]{Bradd:2023}).
    \item A characterization of definable sets (see \cite[Theorem 8.2]{Bradd:2023}).
    \item The existence of saturated models (see \cite[Proposition 7.8]{Bradd:2023}).
\end{itemize}

Finally, we conclude this section with some remarks.

\begin{remark} \label{rmk:elementary-chain}
    Let $(M_\alpha)_{\alpha < \gamma}$ be an elementary chain of $\cL$-structures,
    and let $M = \bigcup_{\alpha < \gamma} M_\alpha$.
    The continuous analogue of Tarski's elementary chain theorem states that the completion of $M$ is an elementary extension of $M_\alpha$ for all $\alpha < \gamma$.
    However, if $\cf(\gamma) > \aleph_0$, then $M$ is automatically complete.
    Therefore, $M$ itself is an elementary extension of $M_\alpha$ for all $\alpha < \gamma$.
\end{remark}

The following result is not stated explicitly in \cite{Bradd:2023},
so we include its statement and proof here.

\begin{lemma} \label{lem:chara-of-type}
    Let $M$ be an $\cL$-structure and let $A \subseteq M$.
    For a set $t(\bar x)$ of $\cL_A$-formulas with free variables $\bar x$, the following are equivalent:
    \begin{enumerate}
        \item $t(\bar x)$ is (the kernel of) a type over $A$;
              that is, there exist a model $N \succeq M$ and a tuple $\bar b \in N$ such that
              $\phi^N(\bar b) = 0$ for all $\phi(\bar x) \in t(\bar x)$.
        \item $t(\bar x)$ is approximately finitely satisfiable in $M$.
              i.e., for every finite subset $t_0(\bar x) \subseteq t(\bar x)$ and every $\ep > 0$, there exists $\bar b \in M$ such that $\abs{\phi^M(\bar b)} < \ep$ for all $\phi(\bar x) \in t_0(\bar x)$.
    \end{enumerate}
\end{lemma}

\begin{proof}
    The implication $(2) \implies (1)$ follows immediately from the compactness theorem for continuous model theory (see \cite[Theorem 6.4]{Bradd:2023}).

    We show (1) $\implies$ (2).
    By (1), there exists a model $N \succeq M$ such that $t(\bar x)$ is realized in $N$.
    Fix a finite subset $t_0(\bar x) \subseteq t(\bar x)$.
    The $\cL_A$-sentence
    \[
        \inf_{\bar x} \max \set[:]{\abs{\phi(\bar x)}}{\phi(\bar x) \in t_0(\bar x)}
    \]
    belongs to $\Th(N) = \Th(M_A)$.
    Thus, for every $\ep > 0$,
    there exists some $\bar b \in M$ such that $\abs{\phi^M(\bar b)} < \ep$ holds for all $\phi(\bar x)$ in $t_0(\bar x)$,
    which completes the proof of (2).
\end{proof}

%% file: files/main_result.tex
\section{Main Result for General Setting}

In this section, we give a sufficient condition on $\rT$ for \cref{thm:main-result-intro} in the general setting of continuous model theory.

\begin{theorem} \label{thm:main-result-general}
    Assume that $\rT$ has a definable set $X$ relative to $\rT$ and a (possibly definable) formula $\phi(\bar x, \bar y)$ satisfying the following conditions:
    \begin{itemize}
        \item $\phi$ takes discrete values $\{0,1\}$ on $X$.
              That is, for every $M \models \rT$, we have $\phi^M(\bar a, \bar b) \in \{0,1\}$ for all $\bar a, \bar b \in X(M)$.
        \item $\phi$ defines an asymmetric relation on $X$.
              That is, for every $M \models \rT$ and all $\bar a, \bar b \in X(M)$,
              if $\phi^M(\bar a, \bar b) = 0$, then $\phi^M(\bar b, \bar a) = 1$.
        \item For any $M \models \rT$, any $n < \omega$, and any $\bar a_i \in X(M)$ $(i < n)$,
              there exists $\bar b \in X(M)$ such that $\phi^M(\bar a_i, \bar b) = 0$ for all $i < n$.
    \end{itemize}

    Then, under the failure of the \CH,
    there exist models $A, B \models \rT$
    with density characters at most $\conti$ such that
    $A^\cU \not\cong B^\cV$ for all $\cU, \cV \in \beta \omega \setminus \omega$.
\end{theorem}

\begin{proof}
    We take a $\kappa$-saturated model $M$ of $\rT$ for sufficiently large cardinals $\kappa$.

    For $\bar a, \bar b \in X(M)$, define $\bar a \trel \bar b$ by
    $\phi^M(\bar a, \bar b) = 0$.
    Let $\bar a \treleq \bar b$ denote
    $\bar a \trel \bar b$ or $\bar a = \bar b$.

    For ordinals $\alpha \le \conti$, we define $\bar a_\alpha \in M$ and $M_\alpha \preceq M$ inductively that satisfy the following conditions:
    \begin{itemize}
        \item For all $\alpha < \conti$, we have $M_\alpha \preceq M_{\alpha + 1}$ and $\chi(M_\alpha) \le \abs{\alpha} + \aleph_0$.
        \item For all $\alpha < \conti$, we have
              $\bar a_\alpha \in X(M_{\alpha+1})$ and $\bar a \trel \bar a_\alpha$ for all $\bar a \in X(M_\alpha)$.
        \item For every limit ordinal $\gamma \le \conti$, we have
              $M_\gamma = \overline{\bigcup_{\alpha < \gamma} M_\alpha}$.
    \end{itemize}

    One can find a separable elementaly submodel $M_0$ of $M$ by the downward \LST theorem.
    Suppose that $M_\alpha$ has already been defined.
    Consider the set of formulas
    \[
        t(\bar x) \deq \{d(\bar x, X)\} \cup \set[:]{\phi(\bar a, \bar x)}{\bar a \in X(M_\alpha)}.
    \]
    By the assumptions on $\phi$, $t(\bar x)$ is approximately finitely satisfiable.
    Thus, by \cref{lem:chara-of-type}, $t(\bar x)$ is a type over $M$.
    By the saturation of $M$, we can choose $\bar a_\alpha \in M$ realizing $t(\bar x)$.
    By the downward \LST theorem, we can take $M_{\alpha+1} \preceq M$ such that $M_\alpha \cup \{\bar a_\alpha\} \subseteq M_{\alpha + 1}$ and $\chi(M_{\alpha + 1}) \le \chi(M_\alpha) + \aleph_0$.
    For a limit ordinal $\gamma \le \conti$, we note that $M_\gamma \preceq M$ by the Tarski's elementary chain theorem.

    In this setting, the following holds:
    \begin{claim*}
        For every uncountable regular cardinal $\kappa \le \conti$ and every $\cU \in \beta\omega \setminus \omega$, we have
        \[
            \cf\left(X(M_\kappa^\cU), \treleq\right) = \kappa.
        \]
    \end{claim*}

    \begin{proof}[(Proof of Claim)]
        First, we show that $\cf(X(M_\kappa^\cU), \treleq) \le \kappa$.
        Let
        \[
            C \deq \set[:]{[(\bar a_\alpha)_n]}{\alpha < \kappa} \subseteq X(M_\kappa^\cU).
        \]
        Obviously, $\abs{C} \le \kappa$.
        We show that $C$ is cofinal in $(X(M_\kappa^\cU), \treleq)$.
        Let $\bar d \in X(M_\kappa^\cU)$.
        Since definable sets commute with ultraproducts,
        we may assume that $\bar d = [(\bar d_n)_n]_\cU$,
        where $\bar d_n \in X(M_\kappa)$ for all $n < \omega$.
        By the regularity of $\kappa$ and \cref{rmk:elementary-chain},
        for each $n < \omega$ we may choose $\alpha_n < \kappa$ such that $\bar d_n \in X(M_{\alpha_n})$.
        Let $\alpha \deq \sup_{n < \omega} \alpha_n$
        and define $\bar c \deq [(\bar a_\alpha)_{n < \omega}]_\cU \in C$.
        Since $\kappa$ is a regular cardinal,
        we have $\alpha < \kappa$,
        and hence $\bar a_\alpha \in X(M_\kappa)$.
        For all $n < \omega$, we have $\bar d_n \treleq \bar a_\alpha$
        Therefore, by the \Los' theorem, we obtain $\bar d \treleq \bar c$.

        Next, we show that $\kappa \le \cf(X(M_\kappa^\cU), \treleq)$.
        Let $C \subseteq X(M_\kappa^\cU)$ be a subset with $\abs{C} < \kappa$.
        Write $C = \set[:]{\bar c^\alpha}{\alpha < \lambda}$, where $\lambda \deq \abs{C} < \kappa$, and $\bar c^\alpha = [(\bar c^\alpha_n)_n]_\cU$ with $\bar c^\alpha_n \in X(M_\kappa)$ for each $n < \omega$ and $\alpha < \lambda$.
        By the regularity of $\kappa$ and \cref{rmk:elementary-chain},
        for each $n < \omega$ and $\alpha < \lambda$,
        we can take $\beta^\alpha_n < \kappa$ such that $\bar c^\alpha_n \in M_{\beta^\alpha_n}$.
        Let $\beta \deq \sup \set[:]{\beta^\alpha_n}{n < \omega, \alpha < \lambda}$ and $\bar d \deq [(\bar a_\beta)_n]_\cU$.
        Since $\kappa$ is regular and $\lambda < \kappa$, we have $\beta < \kappa$, and hence $\bar a_\beta \in X(M_\kappa)$.
        For each $\alpha < \lambda$ and every $n < \omega$, we have
        $\bar c^\alpha_n \trel \bar a_\beta$, and hence
        $\bar c^\alpha \trel \bar d$ by the \Los' theorem.
        Since $\trel$ is an asymmetric relation, it follows that
        $\bar d \not\treleq \bar c^\alpha$ for all $\alpha < \lambda$.
    \end{proof}

    Assume the \CH fails.
    Then we have $\omega_2 \le \conti$.
    Thus, by the claim above, letting $A \deq M_{\omega_1}$ and
    $B \deq M_{\omega_2}$, we obtain
    \[
        \cf(X(A^\cU), \treleq) = \omega_1 \neq \omega_2 = \cf(X(B^\cV), \treleq)
    \]
    for any $\cU, \cV \in \beta \omega \setminus \omega$, thereby completing the proof.
\end{proof}

The above result can be regarded as a generalization of \cite[Theorem 2.1]{GS1:2023} and the results obtained in \cite[\S 3]{Tsuboi:2022}.

\begin{remark}
    The conditions of \cref{thm:main-result-general} can be described in terms of formulas as follows:
    \begin{itemize}
        \item $\rT \models \sup_{\bar x, \bar y \in X}
                  \min\{\abs{\phi(\bar x, \bar y)}, \abs{1 - \phi(\bar x, \bar y)}\}$.
        \item $\rT \models \sup_{\bar x, \bar y \in X}
                  \max\{0, 1 - \phi(\bar y, \bar x) - \phi(\bar x, \bar y)\}$.
        \item For all $n < \omega$, we have
              \[
                  \rT \models
                  \sup_{\bar x_0 \in X} \dots \sup_{\bar x_{n-1} \in X}
                  \inf_{\bar y \in X}
                  \max_{i < n} \phi(\bar x_i, \bar y).
              \]
    \end{itemize}
    Consequently, since $\rT$ is assumed to be complete,
    if these conditions hold in some model of $\rT$,
    then they hold in every model of $\rT$.
\end{remark}

\begin{remark} \label{rmk:classical-results}
    In the classical setting, condition (1) of \cref{thm:main-result-general} is not needed,
    so conditions (2) and (3) suffice.
    In this case, several known results follow from \cref{thm:main-result-general}:
    \begin{itemize}
        \item $\Th(\bQ,<)$ (\cite[Theorem 2.1]{GS1:2023}), with $\phi(x,y) \equiv x < y$.
              \begin{itemize}
                  \item More generally, any (strictly) directed set satisfies the conditions of \cref{thm:main-result-general}.
              \end{itemize}
        \item The theory of random graphs (\cite{Tsuboi:2022}) : $\phi((x, y), (u, v)) \equiv \lnot (x R v) \land y R u$.
        \item The theory of atomless Boolean algebras (\cite{Goto:2024}) : $\phi((a, b), (a', b')) \equiv (a b' \neq 0 \land a (1 - b') \neq 0) \land (a' b = 0 \lor a' (1 - b) = 0)$.
    \end{itemize}
\end{remark}

\begin{remark}
    It is known that the \KS theorem also holds in the continuous setting (cf. \cite{GK:2022}).
    Therefore, the result of \cref{thm:main-result-general} indicates that the failure of the \CH affects the cardinalities of the index sets of ultrapowers.
\end{remark}

%% file: files/cstar.tex
\section{Proof of \cref{thm:main-result-intro}}

In order to prove \cref{thm:main-result-intro}, it suffices to find a theory of unital $\cstar$-algebras that satisfies the conditions of \cref{thm:main-result-general}.
For a $\cstar$-algebra $A$, $\cP(A)$ denotes the set of all projections in $A$.
For $p \in \cP(A) \setminus \{0\}$, $p$ is said to be \emph{minimal} if, for any $q \in \cP(A)$ with $0 \le q \le p$ one has $q = 0$ or $q = p$.

\begin{lemma} \label{lem:example-of-main-result}
    Let $\rT$ be a theory of unital $\cstar$-algebras $A$ satisfying the following conditions:
    \begin{itemize}
        \item For all $p, q \in \cP(A)$, we have $pq = qp$.
        \item $A$ has no minimal projections.
              That is, for every $p \in \cP(A) \setminus \{0\}$, there exists $q \in \cP(A)$ such that $0 \lneq q \lneq p$.
    \end{itemize}
    Then, there exists a formula that satisfies the conditions in \cref{thm:main-result-general}.
\end{lemma}

\begin{proof}
    Let $X(\cdot) \deq \cP(\cdot) \setminus \{0, 1\}$.
    We note that $X$ is a definable set relative to the theory of unital $\cstar$-algebras (cf. \cite[Example 8.5]{Bradd:2023}).

    Define the formulas $\phi(x, y)$ and $\psi((x, y), (x', y'))$ by
    \begin{align}
        \phi(x, y)
         & \deq \max\{1 - \norm{x y}, 1 - \norm{x (1 - y)}\}, \\
        \psi((x, y), (x', y'))
         & \deq \max\{\phi(x, y'), 1 - \phi(x', y)\}.
    \end{align}
    By the commutativity of projections, both $pq$ and $p(1 - q)$ are projections for all $p, q \in X(A)$.
    Hence, $\phi$ and $\psi$ take only the discrete values $\{0, 1\}$ on $X$.
    Moreover, it is easy to check that $\psi$ is asymmetric on $X^2$.

    To show that $\psi$ satisfies condition (3) of \cref{thm:main-result-general},
    it suffices to prove the following claim:
    \begin{claim*}
        For every $n < \omega$ and $p_i \in X(A)$ $(i < n)$, there exist $q, r \in X(A)$ such that
        \[
            \phi^A(p_i, q) = 0
            \quad\text{and}\quad
            \phi^A(r, p_i) = 1
        \]
        for all $i < n$.
    \end{claim*}
    We now prove this claim.
    Let $n < \omega$ and let $p_i \in X(A)$ for $i < n$.

    For each $\sigma \in 2^n$, we define $p_\sigma \in X(A)$ by
    \[
        p_\sigma \deq \prod_{i < n} \left[(1 - \sigma(i)) p_i + \sigma(i) (1 - p_i)\right].
    \]
    Since $A$ has no minimal projections,
    for each $\sigma \in 2^n$ with $p_\sigma \neq 0$,
    we can choose $q_\sigma \in X(A)$ such that $0 \lneq q_\sigma \lneq p_\sigma$,
    and define $q \deq \sum_{\sigma} q_\sigma$.
    Since the sum of all $p_\sigma$ is equal to the unit $1$,
    there exists $\sigma \in 2^n$ such that $p_\sigma \neq 0$,
    and thus $q \neq 0$.
    Moreover, since the $p_\sigma$ are mutually orthogonal,
    the $q_\sigma$ are also mutually orthogonal,
    and hence $q \in X(A)$.
    For each $i < n$,
    we have $p_i q = \sum_{\sigma : \sigma(i) = 0} q_\sigma \neq 0$ and $p_i (1 - q) = \sum_{\sigma : \sigma(i) = 0} (p_\sigma - q_\sigma) \neq 0$.
    Therefore, $\phi^A(p_i, q) = 0$ for all $i < n$.

    Next, define $r \deq p_\sigma \in X(A)$ for some $\sigma \in 2^n$ with $p_\sigma \neq 0$.
    For each $i < n$, if $\sigma(i) = 0$, then $r = p_\sigma \le p_i$, and hence $r (1 - p_i) = 0$.
    Similarly, if $\sigma(i) = 1$, then $r p_i = 0$.
    Therefore, $\phi^A(r, p_i) = 1$ for all $i < n$.
\end{proof}

The following is our main result.

\begin{theorem} \label{thm:main-result}
    Suppose that $\rT$ satisfies the conditions in \cref{lem:example-of-main-result}.
    Then, under the failure of the \CH,
    there exist unital $\cstar$-algebras $A, B \models \rT$ with density characters at most $\conti$
    such that $A^\cU \not\cong B^\cV$ for all $\cU, \cV \in \beta \omega \setminus \omega$.
\end{theorem}

\begin{proof}
    This immediately follows from \cref{thm:main-result-general} and \cref{lem:example-of-main-result}.
\end{proof}

Finally, we conclude this section by giving a concrete example of \cref{thm:main-result-intro}.

\begin{example} \label{ex:cantor-set}
    Let $X$ be the Cantor space $2^\omega$.
    The $\cstar$-algebra $A = C(X)$ is an abelian unital $\cstar$-algebra with no minimal projections.
    Thus, the theory of $A$ satisfies the conditions in \cref{lem:example-of-main-result}.
    More generally, for any abelian unital $\cstar$-algebra $A = C(X)$, where $X$ is a compact Hausdorff space,
    $A$ has no minimal projections if and only if the Boolean algebra $\clop(X)$ consisting of all clopen subsets of $X$ is atomless.

    A nonabelian example is obtained by tensoring a unital nonabelian $\cstar$-algebra with no nontrivial projections such as the Jiang--Su algebra $\mathcal Z$.
    Namely, for $X$ as above, the theory of $\mathcal Z \otimes C(X)$ satisfies the conditions in \cref{lem:example-of-main-result}.
    %    Moreover, instead of tensoring, one can also construct a nonabelian example by taking a reduced product.
\end{example}

%% file: files/topology.tex
\section{Applications to General Topology}

In this section, we discuss applications of \cref{thm:main-result-intro} to general topology in the commutative case.

\begin{definition}
    Suppose that $X$ is a topological space.
    We say that $X$ is \emph{totally disconnected} if all connected components of $X$ are singleton sets.
    We say that $X$ is \emph{$0$-dimensional} if $\clop(X)$ forms a basis for $X$.
    When $X$ is a compact Hausdorff space, the conditions of being totally disconnected and being $0$-dimensional are equivalent.
    A compact Hausdorff space satisfying these conditions is called a \emph{Stone space}.
\end{definition}

We note that Gelfand duality implies that,
for compact Hausdorff spaces $X$ and $Y$,
$X$ and $Y$ are homeomorphic if and only if $C(X) \cong C(Y)$ as $\cstar$-algebras.
Similarly, Stone duality (see \cite[\S 1.3.1]{Farah:2019}) implies that, for Stone spaces $X$ and $Y$,
$X$ and $Y$ are homeomorphic if and only if $\clop(X) \cong \clop(Y)$ as Boolean algebras.
Moreover, a compact Hausdorff space $X$ is a Stone space if and only if $C(X)$ is an abelian $\cstar$-algebra of real rank zero.

\begin{lemma}[{\cite[Lemma 5.7]{EAGLE:2015}}] \label{lem:ultrapower-of-abel-cstar}
    Let $X$ be a compact Hausdorﬀ space, and $\cU$ be an ultrafilter.
    Then $C(X)^\cU \cong C(\sum_\cU X)$, and $\clop(X)^\cU \cong \clop( \sum_\cU X)$
    where $\sum_\cU X$ is the ultracopower of $X$ and $\clop(X)^\cU$ is the classical model-theoretic ultrapower of Boolean algebra $\clop(X)$.
\end{lemma}

For a topological space $X$, the \emph{weight} of $X$, denoted by $w(X)$, is the least cardinality of a basis for $X$.

By applying \cref{thm:main-result} to $C(2^\omega)$, we obtain the following result:
\begin{corollary} \label{cor:top-application}
    Under the failure of the \CH,
    there exist Stone spaces $X$ and $Y$ satisfying the following properties:
    \begin{enumerate}
        \item $w(X), w(Y) \le \conti$.
        \item There exist a set $I$ and non-principal ultrafilters $\cU, \cV \in \beta I\setminus I$ such that $\sum_\cU X$ and $\sum_\cV Y$ are homeomorphic.
        \item For any $\cU, \cV \in \beta \omega \setminus \omega$, $\sum_\cU X$ and $\sum_\cV Y$ are not homeomorphic.
    \end{enumerate}
\end{corollary}

\begin{proof}
    Suppose that the \CH fails.
    By \cref{thm:main-result} with $M = C(2^\omega)$,
    there exist $A, B \models \Th(M)$ such that $\chi(A), \chi(B) \le \conti$ and $A^\cU \not\cong B^\cV$ for all $\cU, \cV \in \beta \omega \setminus \omega$.
    By the \KS theorem, there exist a set $I$ and non-principal ultrafilters $\cU, \cV \in \beta I \setminus I$ such that $A^\cU \cong B^\cV$.

    Note that $M$ is a unital abelian $\cstar$-algebra of real rank zero.
    Since the class of real rank zero $\cstar$-algebras is elementary (see \cite[Example 2.4.2]{MToCA}),
    $A$ and $B$ are also unital abelian $\cstar$-algebras of real rank zero.
    Thus, there exist Stone spaces $X$ and $Y$ such that $A \cong C(X)$ and $B \cong C(Y)$.

    Since, for every infinite Hausdorff space $Z$, the density character of $C(Z)$ coincides with $w(Z)$ (see \cite[Lemma 1.2]{FM:2021}),
    we have $w(X), w(Y) \le \conti$; that is, condition (1) holds.
    Moreover, conditions (2) and (3) also hold for $X$ and $Y$, by the properties of $A$ and $B$ and by \cref{lem:ultrapower-of-abel-cstar}.
\end{proof}

%% file: files/conclusion.tex
\section{Conclusion}

In conclusion, we discuss generalizations of
\cref{thm:main-result-general} and \cref{thm:main-result}.
We note that if a theory $\rT$ is \emph{stable}
(in the sense of continuous model theory; cf. \cite[Definition 5.2]{FHS2:2014}),
then all ultrapowers of models of $\rT$ with density character at most $\conti$ are isomorphic without assuming the \CH,
since they are $\conti$-saturated (cf. \cite[Theorem 5.6 (1)]{FHS2:2014}).
Thus, we are led to consider the following question:

\begin{question}
    Assume the \CH fails.
    Let $\rT$ be an unstable $\cL$-theory.
    Do there exist $A, B \models \rT$ with $\chi(A), \chi(B) \le \conti$ such that
    $A^\cU \not\cong B^\cV$ for all $\cU, \cV \in \beta \omega \setminus \omega$?
\end{question}

We note that the assumptions of \cref{thm:main-result-general} immediately imply that the theory $\rT$ has the \emph{order property} (cf. \cite[Definition 5.2]{FHS2:2014}),
which is equivalent to $\rT$ being unstable (see \cite[Theorem 5.5]{FHS2:2014}).
It is known that for any unital infinite-dimensional $\cstar$-algebra $A$, the theory $\Th(A)$ has the order property (see \cite[Lemma 5.3]{FHS1:2013}).
Therefore, as a special case of the above question, the following question is also of interest.

\begin{question}
    Assume the \CH fails.
    Let $\rT$ be the theory of a unital infinite-dimensional $\cstar$-algebra.
    Do there exist $A, B \models \rT$ with $\chi(A), \chi(B) \le \conti$ such that
    $A^\cU \not\cong B^\cV$ for all $\cU, \cV \in \beta \omega \setminus \omega$?
\end{question}

In the construction given in the proof of \cref{thm:main-result}, the density characters of $A$ and $B$ are controlled only up to $\le \aleph_1$ and $\le \aleph_2$, respectively.
Therefore, the case where $A$ and $B$ are separable is still unsolved.

\begin{question} \label{ques:sep-case}
    Assume that the \CH fails.
    Does there exist a theory $\rT$ of $\cstar$-algebras
    with separable models $A, B \models \rT$ such that
    $A^\cU \not\cong B^\cV$ for all $\cU, \cV \in \beta \omega \setminus \omega$?
\end{question}

In the classical setting of \cref{ques:sep-case}, Shelah \cite{Shelah:1992} gave a result for the theory of graphs.
By the discussion of \cref{cor:top-application},
the abelian $\cstar$-algebra case of \cref{ques:sep-case} can be reformulated as the following question:

\begin{question} \label{ques:top-case}
    Assume the \CH fails.
    Do there exist Stone spaces $X$ and $Y$ satisfying the following properties?
    \begin{enumerate}
        \item $w(X), w(Y) \le \aleph_0$, that is, $X$ and $Y$ are second countable.
        \item There exist a set $I$ and non-principal ultrafilters $\cU, \cV \in \beta I\setminus I$ such that $\sum_\cU X$ and $\sum_\cV Y$ are homeomorphic.
        \item For any $\cU, \cV \in \beta \omega \setminus \omega$, $\sum_\cU X$ and $\sum_\cV Y$ are not homeomorphic.
    \end{enumerate}
\end{question}

\begin{remark}
    By the Riesz-Markov-Kakutani theorem, for a compact Hausdorff space $X$, $C(X)$ is separable if and only if $X$ is metrizable.
    Hence, in \cref{ques:top-case} the requirement that $X$ and $Y$ are second countable can be replaced with the requirement that $X$ and $Y$ are metrizable.
\end{remark}